\documentclass[11pt]{article}
\usepackage{fullpage, graphicx, amsmath, amsfonts, hyperref}
\pdfoutput=1
\newtheorem{theorem}{Theorem}[section]

\newtheorem{lemma}[theorem]{Lemma}

\def\square{\vbox{\hrule height.2pt\hbox{\vrule width.2pt height5pt \kern5pt
                                   \vrule width.2pt} \hrule height.2pt}}
\def\stopproof{\qquad\square}

\newenvironment{proof}%
      {\par\medbreak\noindent{\bf Proof:}\
                       \ignorespaces}%
      {\stopproof\ignorespaces\medbreak}

      {\par\medbreak\noindent{\bf Proof of #1:}\
                       \ignorespaces}%
      {\stopproof\ignorespaces\medbreak}
     
\def\path{\gamma}

\newcommand{\ignore}[1]{}

\newcommand{\Exp}{\mathbb{E}}

\begin{document}

\title{How \emph{Not} to Win a Million Dollars: \\
A Counterexample to a Conjecture of L. Breiman}

\author{Thomas P. Hayes\thanks{Department of Computer Science, University of New Mexico,
Albuquerque, NM~87108, U.S.A.\ \ Email:
\hbox{hayes@cs.unm.edu}.}
}
\date{}

\maketitle

\begin{abstract}
Consider a gambling game in which we are 
allowed to repeatedly bet a portion of our
bankroll at favorable odds.  We investigate
the question of how to minimize the expected
number of rounds needed to increase our
bankroll to a given target amount.

Specifically, we disprove a 50-year old conjecture
of L. Breiman~\cite{Breiman}, that there exists a
{\em threshold strategy} that optimizes the expected 
number of rounds; that is, a strategy that always 
bets to try to win in one round whenever
the bankroll is at least a certain threshold, 
and that makes Kelly bets (a simple proportional betting scheme)
whenever the bankroll is below the threshold.
\end{abstract}

\setcounter{footnote}{0}
\setcounter{page}{1}

\section{The Conjecture}

Consider a favorable gambling game, 
such as betting at 3:1 odds on the outcome of a fair coin toss.
If we are allowed to play this as many times as we like (decided adaptively), 
we can eventually increase our winnings to any desired 
target amount, with certainty.  For instance, proportional betting
strategies such as the Kelly criterion (see Kelly~\cite{Kelly} or
Breiman~\cite{Breiman}), have long been known to accomplish this.

Breiman~\cite{Breiman} ``hopefully conjectured'' the following.
Suppose our goal is to achieve a set target bankroll, say \$1M, starting
with a fraction $\xi$ of that amount.  Let $T(\xi)$ be the expected
number of rounds we have to play before we attain our goal.
Then there exists a threshold $0 < \xi_0 < 1$, and an optimal strategy of the
following form:
\begin{itemize}
\item When the current bankroll is less than $\xi_0$, bet to optimize
  $\Exp \log($bankroll$)$.
This is achieved by betting
a particular fraction of the current bankroll, which is only dependent
on the proposition being offered.  This is sometimes known as ``Kelly betting.''
\item When the current bankroll is at least $\xi_0$, bet to reach
the target bankroll in the current round.
\end{itemize}
We will call such strategies ``threshold strategies.''

Breiman describes this conjecture as ``expressing a moderate faith
in the simplicity of things.''  Indeed, his proposed strategy
seems quite plausible.  However, our main result is a proof that
this strategy is not optimal.

\begin{theorem} \label{thm:main}
There exists a favorable gambling game and 
initial bankroll $\xi$, for which the optimal $T(\xi)$ is at most
$13/14$ times that of any threshold strategy.
\end{theorem}

We have not attempted here to optimize the constant $13/14$, but it
would be nice to know what its best possible value is.

\section{Preliminaries}

We will use the following version of Doob's Optional Stopping Theorem
(see \cite[Theorem 10.10]{Williams}).

\begin{theorem}[Doob's Optional Stopping Theorem] \label{thm:opt-stop}
Let $(X_t)_{t \ge 0}$ be a supermartingale.
Let $T$ be a stopping time with $\Exp T < \infty$, and suppose
there is a constant $C$ such that, for all $t \ge 1$, $|X_{t+1} - X_t|
\le C$.  Then $X_T$ is integrable and $\Exp X_T \le \Exp X_0$.
\end{theorem}

\section{Counterexample}

For our gambling game, consider a biased coin which comes up heads
with probability $2/3$.  Suppose we are allowed to bet on heads, and 
be paid off at $2:1$ odds.  That is, for each unit bet, the net change
in our bankroll is $-1$ with probability $1/3$, and $+2$ with
probability $2/3$.

In this case, it is easily checked that the Kelly criterion says to
bet $1/2$ of the current bankroll at each timestep, so that the 
new bankroll will be: half the current bankroll with probability $1/3$, 
and twice the current bankroll with probability $2/3$.

For an initial bankroll, $x$, let $T(x)$ be the expected number of 
plays of this game until a bankroll of at least $1$ is achieved,
under the strategy minimizing this quantity

\begin{lemma} \label{lem:3k}
For our example game, when $x = 1/2^k$, for $k$ a positive integer,
the (unique) optimal strategy is to bet $1/2$ of the current bankroll every 
round until a bankroll of $1$ is reached.  Therefore, $T(1/2^k) = 3k$.
\end{lemma}

\begin{proof}
First we observe that the given strategy results in an expectation of
$3k$ rounds until the bankroll reaches $1$.  This is because the 
base-2 logarithm of the bankroll is a biased random walk on the
negative integers, that moves one unit  towards $0$ with probability $2/3$
at each timestep, and one unit away with probability $1/3$.
The analysis of the hitting time to $0$ for this random walk is standard.

To see that this is the best possible result, consider any 
possible strategy, and let $x(t)$ denote the bankroll after $t$ steps.
Let $\tau$ denote the hitting time to the target bankroll:  $\tau := \min \{t : x(t) \ge 1\}$.
Define $Y_t := 3 \log_2(x(t)) - t$.  Observe that, regardless of the
strategy chosen, $(Y_t)$ is a supermartingale,  since Kelly betting 
maximizes the conditional expectation of $Y_{t+1}$ given $Y_t$,
and, for Kelly betting, $Y_t$ would be a martingale.
Further, note that $Y_0 = - 3k$, and $Y_\tau = - \tau$.

Now, we may assume $\Exp \tau < \infty$, since otherwise
this strategy is clearly worse than Kelly betting.  Suppose
for now that $|Y_{t+1} - Y_t|$ is bounded almost surely.
Then applying the Doob's Optional Stopping Theorem~\ref{thm:opt-stop}, 
we have
$\Exp \tau = - \Exp Y_{\tau} \ge - Y_0 = 3k$, which proves
that, again, the Kelly strategy is superior.

Finally, we will show that any strategy with $|Y_{t+1} - Y_t|$ 
unbounded can be strictly improved upon by another strategy with 
$|Y_{t+1} - Y_t| \le 25$, and therefore our previous assumption 
that the steps taken by $Y$ are bounded was made without loss of
generality.  

First note that $Y_{t+1} \le Y_t + \log_2(3) - 1$ absolutely, since even if we bet the entire
bankroll, we cannot more than triple our stake.
Suppose, for some value of $x(t)$, this strategy bets more than $1 - 3^{-15}$ of 
the bankroll.  Let's look at the expected number of rounds
needed until $x(t') \ge 2 x(t)$.  With probability $2/3$, $t' = t+1$,
but with probability $1/3$, the bankroll initially drops by a 
factor of $3^{15}$, and hence, no matter what, it 
will take more than $15$ rounds to return to
its initial value of $x(t)$.  Thus, in expectation, it takes more than
$6$ rounds to exceed $x(t)$ for the first time.  And, moreover, 
the value of $x(t')$ is in the interval $(x(t), 3x(t))$.

Now, note that Kelly betting quadruples the stake in an
expected $6$ rounds, which clearly dominates the above
strategy.  So, it would be strictly superior to ``bet to double'' until
the stake reaches $4 x(t)$, and then proceed optimally from
that point onward.  Thus, any strategy that ever bets more than
$1 - 3^{-15}$ of its bankroll can be improved upon by one that does
not, which gives us bounds of 
$-15 \log_2(3) - 1 \le Y_{t+1} - Y_t \le \log_2(3) - 1$ on the 
revised strategy, which completes the proof.
\end{proof}



Lemma~\ref{lem:3k} is interesting because it shows that Kelly betting
is in fact optimal for infinitely many starting bankrolls.
Furthermore, when coupled with our next two (easy) lemmas, this actually
implies that there exist threshold strategies for which $T(\xi)$ is within
a constant factor of optimal.  These results are not specific to the
example game chosen for this paper; precise statements and proofs
are left as an exercise to the reader.

On the flip side, as we will see, 
Lemma~\ref{lem:3k} is also the key to our proof that every threshold 
strategy is actually suboptimal.

\begin{lemma} \label{lem:decr} \label{lem:3/2lb}
In general, $T(x)$ is a decreasing function of $x$. 
Furthermore, for our example game, $T(x) > 3/2$ for all $x < 1$.
\end{lemma}

\begin{proof}
Clearly $T(x)$ is non-increasing, since extra money can always be
ignored with no penalty.  We omit a detailed proof that $T$ is
strictly decreasing, noting only if there is extra money being ignored,
then after a sufficiently long sequence of consecutive losses under, say, a proportional
betting scheme, this extra money will become the vast majority of the
bankroll, at which point we can appeal to Lemma~\ref{lem:3k} to see
that the expected hitting time to $1$ is strictly better than without
the extra money.

To see that $T(x) > 3/2$, note that we cannot achieve a bankroll of
$1$ without winning at least one coin toss.  But the expected number
of coin tosses until the first heads is $1/(2/3) = 3/2$.  So this is
clearly a lower bound on the hitting time to $1$.  The inequality is
strict, since if we keep losing,
eventually our bets must become too small to 
guarantee reaching $1$ on the first heads.
\end{proof}

\begin{lemma} \label{lem:bet-to-1}
For our example game, any optimal strategy always bets to $1$ when the
bankroll is $x \ge 1/2$.  Additionally, $T(2/3) = 2$ and $T(7/9) =
5/3$.
\end{lemma}

\begin{proof}
Suppose for a bankroll $x \in [1/2,1)$, and our strategy bets to some value $y < 1$.
Then, since the first coin flip either results in a bankroll of $y$,
or a bankroll $< x$, we have
\begin{align*}
T(x) &= 1 + \frac{2}{3} T(y) + \frac{1}{3}T(x - (y-x)/2) \\
&> 1 + \frac{2}{3} T(y) + \frac{1}{3}T(x) \\
&> 2 + \frac{1}{3}T(x) & \mbox{By Lemma~\ref{lem:3/2lb}}
\end{align*}
This implies $T(x) > 3$.  But, by Lemma~\ref{lem:3k}, we know that
$T(1/2) = 3$, so we have a contradiction to the fact that $T$ is a decreasing function
(Lemma~\ref{lem:decr}).
Thus the correct strategy must be to bet to $1$. 

Since, from $x=2/3$ or $x=7/9$, the strategy of betting to $1$ 
either results in winning directly, or reaching a bankroll of $1/2$,
and since Lemma~\ref{lem:3k} tells us that $T(1/2) = 3$, an easy
calculation yields the values of $T(2/3)$ and $T(7/9)$, for which we
shall have a use later.
\end{proof}

\begin{lemma} \label{lem:1/3-1/2}
For our example game, if Breiman's conjecture were true, then the
critical threshold $\xi_0$ would necessarily be in $(1/3, 1/2]$.
\end{lemma}

\begin{proof}
Since winning bets are paid off at $2:1$ odds, it is impossible to
``bet to $1$'' with a bankroll of less than $1/3$.  Moreover,
unless the bankroll is strictly greater than $1/3$, we cannot bet
to $1$ without risking the entire bankroll, in which case
$\Exp \tau = + \infty$.  So $\xi_0 > 1/3$.
On the other hand, if $x(t) > 1/2$, then betting the Kelly criterion
(risking half the bankroll) is too much, as winning results in a
bankroll exceeding $1$, which has no added utility.  So $\xi_0 \le 1/2$.
(Lemma~\ref{lem:bet-to-1} also implies $\xi_0 \le 1/2$.)
\end{proof}

\begin{lemma} \label{lem:7/18}
For any $\xi_0 \in (1/3, 1/2]$, playing Breiman's strategy results in $T(7/18) = 14/3$.
\end{lemma}

\begin{proof}
We consider two cases.  {\bf Case A:}\   $\xi_0 \le 7/18$.  In this case, we
first bet to $1$, winning the game with probability $2/3$, and
otherwise losing $11/36$, for a new bankroll of $1/12$.
Assuming we lost, we will now risk half the bankroll at each step,
until our bankroll has doubled up to a value of $2/3$.  (Note that,
by Lemma~\ref{lem:1/3-1/2}, $\xi_0 > 1/3$.)  Note that, by 
Lemma~\ref{lem:3k}, doubling up thrice takes $9$ rounds in
expectation.  Since Lemma~\ref{lem:bet-to-1} tells us $T(2/3) = 2$,
we can now compute
\[
T(7/18) = 1 + \frac13 \left(9 + T(2/3)\right) = \frac{14}{3}.
\]

\noindent{\bf Case B:}\  $\xi_0 > 7/18$.  In this case, we first bet to double up to
$7/9$, which, by Lemma~\ref{lem:3k} takes 3 rounds in expectation.
Since Lemma~\ref{lem:bet-to-1} tells us $T(7/9) = 5/3$, we therefore
have
\[
T(7/18) = 3 + T(7/9) = \frac{14}{3},
\]
just as in Case A, completing the proof.
\end{proof}

Combining Lemmas~\ref{lem:1/3-1/2} and \ref{lem:7/18} shows us that,
if Breiman's conjecture were true, then $T(7/18) = 14/3$ for our example
game.  Our next result contradicts this.

\begin{lemma}
For our example game, $T(7/18) \le 13/3$.
\end{lemma}

\begin{proof}
On the first bet, suppose we bet $5/36$.  If we win, our
bankroll goes up to $2/3$.  
If we lose, our bankroll goes down to $1/4$.
Since Lemmas~\ref{lem:3k} and \ref{lem:bet-to-1} tell us
that with optimal play, 
$T(1/4) = 6$ and $T(2/3) = 2$, we therefore have
\[
T(7/18) \le 1 + \frac23 T(2/3) + \frac13 T(1/4) 
= 1 + \frac23 \; 2 + \frac13 \; 6
= \frac{13}{3}. 
\]
\end{proof}

Combining the above Lemmas shows us that Breiman's conjecture is
false, and with a tiny bit more work, that any strategy meeting
Breiman's form has an expected cost at least $14/13$ of optimal
starting from the initial bankroll of $7/18$, 
thus proving Theorem~\ref{thm:main}

It would be nice to know the ``price'' for playing a threshold
strategy.  We have seen that this is at least $14/13$ for our example
game.  Although we have made no serious attempt to improve this value, it 
seems likely that it is not far off the mark.

\begin{figure}[ht]
\centering
\includegraphics[width=17cm]{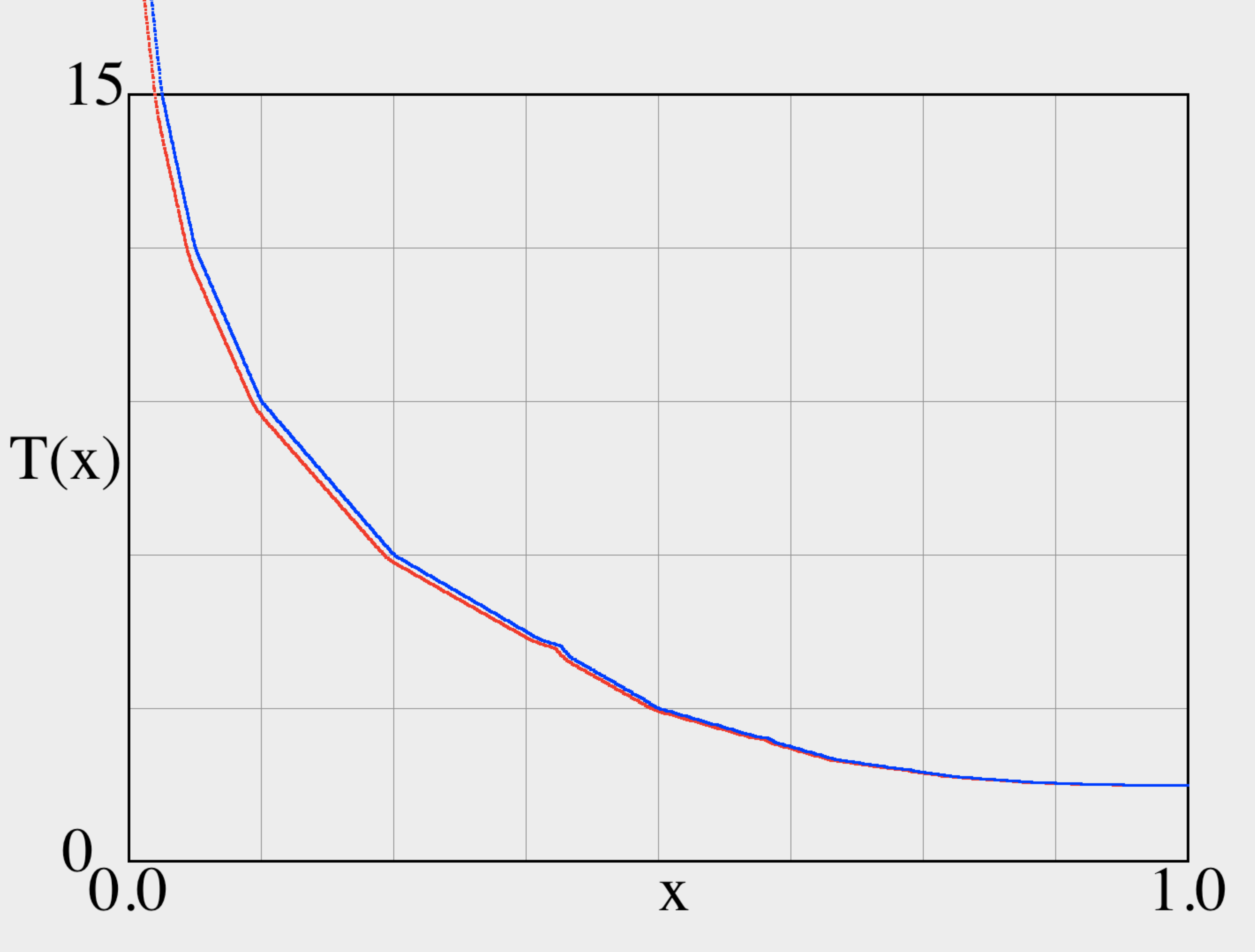}
\caption{Computer-generated upper and lower bounds on the optimal 
  expected number of rounds
  needed to reach a bankroll of $1$ from an initial bankroll of $x$,
  in our example game.  
}
\label{fig:BCopt}
\end{figure}

Figure~\ref{fig:BCopt} shows a graph of upper and lower bounds on the
optimal $T(x)$.  It was obtained by a computer program, which,
beginning with the upper and lower bounds that follow immediately
from Lemma~\ref{lem:3k} and \ref{lem:decr}, recursively derives better
and better upper and lower bounds in terms of the previous ones.
All of the bounds the program works with are step functions, which
allow for exact calculations, up to the precision of the machine arithmetic.
For efficiency reasons, the program frequently approximates these
step functions with more conservative ones, so as to avoid an
exponential growth in the number of distinct function values it has to
keep track of. 

Note that the upper bound is very close to all the values we know
exactly from our lemmas, suggesting that the lower bound remains
too conservative.

Also note that the ``bump'' visible in the graph at $x \approx 0.41$
is not an artifact.  For comparison, $7/18 = 0.388\overline{8}$.
We have no good explanation for this feature of the graph, but it
seems to suggest that optimal betting has some interesting
structure which remains to be understood.

\section*{Acknowledgments}

I would like to thank Yuval Peres and Evangelos Georgiadis for
introducing me to this problem and for helpful comments.

\end{document}